\documentclass[notitlepage,11pt]{article}
\usepackage{amssymb,amsmath,comment}
\usepackage{a4wide}
\catcode`\@=11 \@addtoreset{equation}{section}

\catcode`\@=12
\usepackage{colortbl}%

\newcommand{\fa} {\forall}

\newcommand{\al} {\alpha}
\newcommand{\ba} {\beta}
\newcommand{\de} {\delta}
\newcommand{\ga} {\gamma}

\newcommand{\Om} {\Omega}
\newcommand{\ra} {\rightarrow}

\newcommand{\De} {\Delta}
\newcommand{\la} {\lambda}

\newcommand{\noi} {\noindent}

\newcommand{\mb} {\mathbb}
\newcommand{\mc} {\mathcal}

\def\QED{\hfill {$\square$}\goodbreak \medskip}

\newtheorem{Theorem}{Theorem}[section]
\newtheorem{Lemma}[Theorem]{Lemma}

\newtheorem{Definition}[Theorem]{Definition}

\begin{document}
\title
{On the solvability of resonance problems for nonlocal elliptic equations}

\author{
{\bf  Sarika Goyal\footnote{email: sarika1.iitd@gmail.com}}\\ 
{\small Department of Mathematics}, \\{\small Indian Institute of Technology Delhi}\\
{\small Hauz Khas}, {\small New Delhi-16, India}\\
 }

\date{}

\maketitle

\begin{abstract}

In this article, we consider the following problem:
$$ \quad \left\{
\begin{array}{lr}
 \quad  (-\De)^s u = \al u^+ -\ba u^{-} + f(u) + h \; \text{in}\;\Om\\
 \quad \quad \quad \quad u =0 \; \text{on}\; \mb R^n\setminus \Om,
\end{array}
\right.
$$
where $\Om\subset \mb R^n$ is a bounded domain with Lipschitz
boundary, $n> 2s$, $0<s<1$, $(\al,\ba)\in \mb R^2$, $f: \mb R\ra \mb
R$ is a bounded and continuous function and $h\in L^2(\Om)$. We
prove the existence results in two cases: First, the nonresonance
case, where $(\al,\ba)$ is not an element of the Fu\v{c}ik spectrum.
Second, the resonance case, where $(\al,\ba)$ is an element of the
Fu\v{c}ik spectrum. Our existence results follows as an application
of the Saddle point Theorem. It extends some results, well known for
Laplace operator, to the nonlocal operator.

\medskip

\noi \textbf{Key words:} Nonlocal problem, Fu\v{c}ik spectrum, Resonance, Saddle point Theorem.

\medskip

\noi \textit{2010 Mathematics Subject Classification:} 35A15, 35B33,
35H39

\end{abstract}

\bigskip
\vfill\eject

\section{Introduction}
\setcounter{equation}{0}
Let $s\in (0,1)$ and let $\Om\subset \mb R^n$ is a bounded domain with Lipschitz boundary, $n>2s$.
We consider the following problem:
$$ \quad \left.
\begin{array}{lr}
 \quad  (-\De)^s u= \al u^+ -\ba u^{-} + f(u) + h \; \text{in}\;\Om,
 \quad  \; u =0\;\quad \text{on}\; \mb R^n\setminus \Om,
\end{array}
\right.
$$
where $(\al,\ba)\in \mb R^2$, $f: \mb R\ra \mb R$ is a bounded and
continuous function, $h\in L^2(\Om)$ and  $u^{\pm}= \max\{\pm
u,0\}$. Here, $(-\De)^{s}$ is the fractional Laplacian operator
defined as
\begin{equation*}
(-\De)^{s} u(x)= -\frac{1}{2} \int_{\mb
R^n}\frac{u(x+y)+u(x-y)-2u(x)}{|y|^{n+2s}} dy \;\text{for all} \;
x\in \mb R^n.
\end{equation*}

\noi In general, we study the corresponding problem driven by the non-local operator $\mc L_{K}$ is
$$ (P_\la) \quad \left\{
\begin{array}{lr}
 \quad - \mc L_K u = \al u^+ -\ba u^{-} + f(u) + h \; \text{in}\;\Om,
 \quad  \; u =0\;\quad \text{on}\; \mb R^n\setminus \Om,
\end{array}
\right.
$$
where
the nonlocal operator $\mc L_K$ is defined as
\[\mc L_K u(x):= \frac12 \int_{\mb R^n}(u(x + y) + u(x- y) -2u(x))K(y)
dy\;\;\text{for\; all}\;\; x\in \mb R^n.\]
Here we assume that the function $K :\mb R^n\setminus\{0\}\ra(0,\infty)$ satisfies the following:
 \begin{enumerate}
 \item[(K1)] $mK \in L^1(\mb R^n),\;\text{where}\; m(x) = \min\{|x|^2,
1\}$,
\item[(K2)] There exist $\la>0$ and $s\in(0,1)$ such that $K(x)\geq \la
|x|^{-(n+2s)},$
\item[(K3)] $K(x) = K(-x)$ for any $ x\in \mb R^n\setminus\{0\}$.
\end{enumerate}
\noi In case $K(x) = |x|^{-(n+2s)}$, $\mc L_K$ is the fractional
Laplace operator $-(-\De)^s$. When $s=1$, the fractional Laplacian
operator becomes the usual Laplace operator. There has been done a
lot of work related to the solvability of resonance problem with
respect to spectrum, Fu\v{c}ik spectrum for Laplace equation
 see \cite{dr,dr1, LL, EM} and references therein. The Fu\v{c}ik spectrum in the case of
Laplacian, $p$-Laplacian equation with Dirichlet boundary condition
has been studied by many authors \cite{CC, cfg, fg}.

Recently a lot of attention is given to the study of fractional and
non-local equations of elliptic type due to concrete real world
applications in finance, thin obstacle problem, optimization,
quasi-geostrophic flow etc. Dirichlet boundary value problem in case
of fractional Laplacian with polynomial type nonlinearity using
variational methods is studied in \cite{mp,var}. Fiscella, Servadei
and Valdinoci in \cite{fsv} studied the resonance problem with
respect to the spectrum for non local equation.  To the best of our
knowledge, no work has been done related to the solvability of
resonance problem with respect to the Fu\v{c}ik spectrum for non
local equation.

\noi The Fu\v{c}ik spectrum of the non-local operator $\mc L_{K}$ is
defined as the set $\sum_{K}$ of $(\al,\ba)\in \mb
 R^2$ such that
\begin{equation}\label{eq01}
 \quad \left.
\begin{array}{lr}
 \quad -\mc L_{K}u = \al u^{+} - \ba u^{-} \; \text{in}\;
\Om,
  \quad \; u = 0 \;\quad \mbox{on}\; \mb R^n \setminus\Om,\\
\end{array}
\quad \right.
\end{equation}
\noi has a nontrivial solution $u$.
For $\al=\ba=\la$, the Fu\v{c}ik spectrum of \eqref{eq01} becomes the usual spectrum of $\mc L_K$. In this case, $u$ satisfies
\begin{equation}\label{eq02}
 \quad \left.
\begin{array}{lr}
 \quad -\mc L_{K}u = \la u \; \text{in}\;
\Om, \quad \; u = 0 \;\quad \mbox{on}\; \mb R^n \setminus\Om.\\
\end{array}
\quad \right.
\end{equation}
 Let $0<\la_1<\la_2\leq...\leq\la_k\leq...$ denote the sequence of eigenvalues of
  \eqref{eq02} and $\{\phi_k\}_k$ denote the sequence of eigenfunctions corresponding to $\la_k$. Then it is proved in \cite{var} that the first eigenvalue $\la_1$ of \eqref{eq02} is simple,
isolated and can be characterized as follows
\[\la_1 = \inf_{u\in X_{0}}\left\{\int_{Q}(u(x)-u(y))^2 K(x-y)dxdy : \int_{\Om} u^2=1\right\}.\]
The author also proved that the eigenfunctions corresponding to
$\la_1$ is non-negative. Moreover, one can observe that $\sum_{K}$ clearly
contains $(\la_k,\la_k)$ for each $k\in \mb N$ and two lines
$\la_1\times\mb R$ and $\mb R\times\la_1$. $\sum_{K}$ is symmetric
with respect to the diagonal. In \cite{ss}, it is shown that the two
lines $\mb R\times \la_1$ and $\la_1\times \mb R$ are isolated in
$\sum_{K}$ and the second
eigenvalue $\la_2$ of $-\mc L_K$ has a variational characterization. But here we will characterize a portion of $\sum_K$ using the variational method. That is, the eigenvalue pair will be obtained as minima or minimax values of an appropriate functional.

In the homogeneous case, where $\al=\ba=\la$ and $f\equiv 0$, the solvability of $(P_{\la})$ can be completely described by the Fredholm Alternative, which says that if $\la$ is not an eigenvalue of $-\mc L_K$, then the problem has a unique solution for any $h$, and if $\la$ is an eigenvalue of $-\mc L_K$, then the problem $(P_{\la})$ has a solution if and only if $h$ is orthogonal to the corresponding eigenspace.

For the nonhomogeneous case, where $\al=\ba=\la$ and $f\ne 0$, Fiscella, Servadei and Valdinoci in \cite{fsv}, studied the existence results for the following problem
\begin{equation}\label{ee1}
 \quad \left\{
\begin{array}{lr}
 \quad - \mc L_K u + q(x)u = \la u + f(u) + h(x) \; \text{in}\;\Om,
 \quad  u =0\;\text{on}\; \mb R^n\setminus \Om,
\end{array}
\right.
\end{equation}
where
$f$, $q$ and $h$ are sufficiently smooth functions.
They showed that if $\la$ is not an eigenvalue(nonresonance), then it has a solution with no further restriction on $f$ and $h$, and if $\la$ is an eigenvalue(resonance), then they need some extra conditions on  $f$ and $h$. Precisely, denoting by
\[f_l=\lim_{t\ra-\infty} f(t)\;\mbox{and}\; f_r=\lim_{t\ra\infty} f(t),\]
they assume that $f_l$ and $f_r$ exist, are finite and such that $f_l>f_r$
and
\[f_r\int_{\Om} \phi^{-}(x)dx- f_l\int_{\Om} \phi^{+}(x)dx <\int_{\Om} h(x)\phi(x) dx< f_l \int_{\Om} \phi^{-}(x)dx -f_r\int_{\Om} \phi^{+}(x)dx  \]
for any nontrivial $\phi$ in the eigenspace associated with $\la$.
\noi We would remark that these extra conditions on $f$ and $h$ are exactly the same required in the resonant setting, when dealing with the classical Laplace operator. Moreover, in the resonant case for fractional Laplacian, they are able to treat this case only if
$\la$ satisfies the following condition:
\begin{align*}
\la\; & \mbox{is an eigenvalue of}\; -\mc L_K +q\;\mbox{ such that all the eigenfunctions corresponding to}\\
 &\la\;\mbox{ have nodal set with zero Lebesgue measure}.
\end{align*}
 The nodal set of a function $g$ in $\Om$ is the level set $\{x\in \Om: g(x)=0 \}$. For example, in case of fractional Laplacian this condition is true when $\la$ is its first eigenvalue. Moreover this condition is compatible with the classical Laplace operator, in this context it is satisfied by every eigenvalue.

In this paper, we studied the problem $(P_{\la})$ with respect
to the Fu\v{c}ik spectrum for nonlocal equation. Here we use the
variational argument which was developed by Castro and Chang in
\cite{CC} for the Laplace operator. One can easily extend some
results for Laplace equation to nonlocal equation. But for
completeness, we provide the details of the proof.\\
 Now for the nonresonance case,
we assume that $\al$ lies strictly between consecutive eigenvalues
of $(-\De)^s$, call them as $\la_k<\la_{k+1}$, and
 we also assume that $\al\leq \ba<\ba(\al)$, where $\{(\al,\ba): \al\leq \ba<\ba(\al)\}$ contains no points in $\sum_K$, according to the Castro-Chang
characterization in case of Laplace operator. We note that one can
also explore the similar characterization for nonlocal operator. Now
we prove the following:
\begin{Theorem}\label{th1}
Assume $\la_k<\al<\la_{k+1}$, $\al\leq \ba<\ba(\al)$, $f:\mb R\ra \mb R$ is a bounded and continuous function, and $h\in L^2(\Om)$, then the problem $(P_\la)$ has at least one weak solution.
\end{Theorem}

\noi In the resonance case, we still assume that
$\la_k<\al<\la_{k+1}$, but now assume that $\ba=\ba(\al)$, as above,
where $(\al,\ba(\al))\in \sum_{K}$. The solvability condition that
we impose is the following:

Let $F(u):= \int_{0}^{u} f(t) dt$. If $\{u_k\}\in X_0$ such that $\|u_k\|_{L^2}\ra \infty$ and $\frac{u_k}{\|u_k\|_{L^2}}$ converges in $L^2(\Om)$ to some $v$, a nontrivial Fu\v{c}ik eigenfunction associated with $(\al,\ba)$, then
\[(GLL):\quad\lim_{k\ra\infty}\int_{\Om} (F(u_k)+h u_k) = -\infty.\]
This condition is known as the generalization of Landesman-Lazer
condition.
\begin{Theorem}\label{th2}
Assume $\la_k<\al<\la_{k+1}$, $\ba=\ba(\al)$, $f:\mb R\ra \mb R$ is bounded and continuous, and $h\in L^2(\Om)$ and $(GLL)$ is satisfied.
Then the problem $(P_\la)$ has at least one weak solution.
\end{Theorem}
\section{Preliminaries}
In this section we will recall function spaces which is introduced by Servadei and some standard results from Functional analysis and critical point Theory.
\noi In \cite{mp}, Servadei and Valdinoci discussed the Dirichlet
boundary value problem in case of fractional Laplacian using the
Variational techniques. We also use similar variational
techniques to find the existence result for $(P_{\la})$. Due to non-localness of the fractional
Laplacian, we use the function spaces introduced by Servadei.
\[X= \left\{u|\;u:\mb R^n \ra\mb R \;\text{is measurable},\;
u|_{\Om} \in L^2,\;  \left(u(x)- u(y)\right)\sqrt{K(x-y)}\in
L^2(Q)\right\},\]
\noi where $Q=\mb R^{2n}\setminus(\mc C\Om\times \mc C\Om)$ and
 $\mc C\Om := \mb R^n\setminus\Om$. The space X is endowed with the norm
\begin{align*}\label{eq44}
 \|u\|_X = \|u\|_{L^2(\Om)} +\left( \int_{Q}|u(x)-u(y)|^{2}K(x-y)dx
dy\right)^{\frac12}.
\end{align*}
 Then $X_0$ is  defined as
 \[ X_0 = \{u\in X : u = 0 \;\text{a.e. in}\; \mb R^n\setminus \Om\}\]
equipped with the norm
\[\|u\|=\left( \int_{Q}|u(x)-u(y)|^{2}K(x-y)dx
dy\right)^{\frac12}\] and the space
\[L^2(\Om):= \{u:\Om\ra \mb R: u \;\mbox{is measurable},\;\int_{\Om}u^2 dx <\infty\}\]
endowed with the norm
\[\|u\|_{L^2}=\left(\int_{\Om}u^2dx \right)^\frac12\]
are both Hilbert spaces. Note that the norm
$\|.\|$ on the space $X_0$ involves the interaction between $\Om$ and $\mb
R^n\setminus\Om$. For more details on these function spaces and the embedding theorems, we  refer to \cite{hi, mp}.
%
\begin{Definition}\label{def1}
A function $u \in X_0$ is a weak solution of $(P_\la)$, if for every $v\in X_0$, $u$
satisfies
\begin{align*}
\int_{Q}(u(x)-u(y))(v(x)-v(y)) K(x-y)dxdy =  \int_{\Om} \left(\al
u^{+}
 - \ba u^{-} + f(u)+h\right) v dx.
\end{align*}
\end{Definition}
\noi Now we denote $X_1= span[\phi_1, \phi_2,\cdots,\phi_k]$. That
is, the linear span of the first $k$ eigenfunctions, and $X_2:=
X_{1}^{\perp}=[\phi_{k+1},\phi_{k+2},\cdots]$. The sequence
$\{\phi_k\}_{k\in\mb N}$ of eigenfunctions is an orthonormal basis
of $L^{2}(\Om)$ and an orthogonal basis of $X_0$. By definition, the
subspaces $X_1$ and $X_2$ are orthogonal and $X_0= X_1\oplus X_2$.
The Fourier expansion of a function $u\in X_{0}$ is
\[u=\sum_{j=1}^{\infty}c_j \phi_j,\]
and note that
\begin{align*}
\int_{Q}|u(x)-u(y)|^2 K(x-y) dx dy= \sum_{j=1}^{\infty} \la_{j}
c_{j}^2\;\quad\mbox{and}\;\quad\int_{\Om}u^2 dx =
\sum_{j=1}^{\infty} c_{j}^{2}.
\end{align*}
This has helpful consequence such as
\begin{align*}
\int_{Q}|u(x)-u(y)|^2 K(x-y) dx dy&\leq \la_k \int_{\Om} u^2 dx \;\fa\; u\in X_1\;\mbox{and}\\
\int_{Q}|v(x)-v(y)|^2 K(x-y) dx dy &\geq \la_{k+1} \int_{\Om} v^2dx
\;\fa\; v\in X_2.
\end{align*}

\noi To analyze problem \eqref{eq01} we consider the functional
\begin{align*}
J_{\al,\ba}(u)= \frac{1}{2}\left(\int_{Q}|u(x)-u(y)|^2 K(x-y) dx dy
-\al\int_{\Om} |u^+|^2 dx-\ba\int_{\Om}|u^-|^2 dx\right),
\end{align*}
which is $C^{1}$ functional on $X_0$ with
{\small\begin{align*}
\langle J_{\al,\ba}^{\prime}(u), v \rangle= \int_{Q}(u(x)-u(y))(v(x)-v(y)) K(x-y) dx dy -\al\int_{\Om} u^+ v dx +\ba\int_{\Om}u^- v dx.
\end{align*}}
One can easily see that the critical points of $J_{\al,\ba}$ are weak solutions of \eqref{eq01}.
It will be very helpful to think of $J_{\al,\ba}$ as a $C^1$ functional on $\mb R^2\times X_0$. That is,
\[J: \mb R^2\times X_0\ra \mb R: J(\al,\ba,u):= J_{\al,\ba}(u),\]
with the derivative given by
\begin{align*}
\langle J^{\prime}(\al,\ba,u),(s,t,v)\rangle &= \int_{Q}(u(x)-u(y))(v(x)-v(y)) K(x-y) dx dy -\al\int_{\Om} u^+ v \\
&\quad\quad + \ba\int_{\Om}u^- v- s\int_{\Om}(u^{+})^2 -t\int_{\Om}(u^-)^2.
\end{align*}
It is clear that $\|DJ\|_{(\mb R^2\times X_0)^*}$ is bounded on bounded subsets of $\mb R^2\times X_0$, and so $J$ is uniformly Lipschitz continuous on any bounded subset of $\mb R^2\times X_0$.

\noi To analyze problem $(P_{\la})$ we consider the functional
{\small\[E_{\al,\ba}(u)= \frac{1}{2}\left(\int_{Q}|u(x)-u(y)|^2 K(x-y) dx dy -\al\int_{\Om} |u^+|^2 -\ba\int_{\Om}|u^-|^2\right)- \int_{\Om} (F(u)+hu),\]}
where $F(u):= \int_{0}^{u}f(t)dt$. $E_{\al,\ba}$ is also a $C^{1}$ functional on $X_0$ with
{\small\[\langle E_{\al,\ba}^{\prime}(u), v\rangle = \int_{Q}(u(x)-u(y))(v(x)-v(y)) K(x-y) dx dy - \int_\Om \left(\al u^+ -\ba u^- +f(u)+h \right)v dx.\]}
It is straight forward to see that critical points of $E_{\al,\ba}$ are weak solutions of $(P_{\la})$.\\
\noi To prove the existence of critical points we will use the
following Saddle point Theorem.
\begin{Theorem}
(Saddle Point Theorem:) Let $F: X_0\ra \mb R$ be a $C^1$ functional which satisfies $(PS)$. Assume that there are sets $\mc X_1$, $\mc X_2\subset X_0$ such that
\begin{enumerate}
\item[$(i)$] $\mc X_1=\tilde{\ga}(S^{k-1})$, where $\tilde{\ga}:S^{k-1}\ra X_0$ is continuous.
\item[$(ii)$]$\mc X_2$ links with $\mc X_1$, i.e. if $B$ is the unit ball in $\mb R^k$ and $\ga: B\ra X_0$ is a continuous function such that $\ga\equiv \tilde{\ga}$ on $S^{k-1}$, then $\ga(B)\cap \mc X_2\neq\emptyset$.
\item[$(iii)$] $\sup_{x\in \mc X_{1}} F(x)<\inf_{y\in \mc X_2} F(y)$.
\end{enumerate}
Then
\[c:=\inf_{\ga\in \Gamma}\sup_{x\in B} F(\gamma(x))\]
is a critical point of $F$, where $\Gamma=\{\ga: B\ra X_0: \ga\;\mbox{is continuous and}\;\ga\equiv \tilde{\ga}\;\mbox{on}\; S^{k-1}\}$.
\end{Theorem}

\section{The variational characterization of Fu\v{c}ik Spectrum}

In all that follows we assume that $\la_k<\al<\al_{k+1}$ and the points of $\sum_K$ that we characterize will all lie in this vertical strip in the $(\al,\ba)$ plane. We assume that $\al\leq \ba$, and note that opposite case can be treated via symmetric arguments. Our approach to finding critical points of $J_{\al,\ba}$ will be take advantage of concavity to maximize in the $X_1$ direction, and then to use weak lower semicontinuity to minimize in the $X_2$ direction.

\subsection{Maximizing in the $X_1$ direction}
In this subsection, we will show that the functional $J_{\al,\ba}$ attains a maximizer in $X_1$ direction and the properties of the maximizer function.
First, we prove the general inequality that is used to prove the concavity of the functional in $X_1$ direction.
\begin{Lemma}\label{le1}
Let $(\al_i,\ba_i)\in \mb R^2$ for $i=1,2$, be points satisfying $\al_i\leq \ba_i$, and let $s_i= \ba_i-\al_i$. Let $u_i\in X_1$ and $v_i\in X_2$ for $i=1,2.$ Then there exist a $\de>0$ such that
{\begin{align}\label{e1}
\langle (J^{\prime}_{\al_2, \ba_2}&(u_2+v_2)- J^{\prime}_{\al_1, \ba_1}(u_1+v_1)), (u_2 - u_1) \rangle \notag\\
&\leq -\de\|u_2-u_1\|^2 +|\ba_2-\al_2|(\|u_2-u_1\|_{L^2}+\|v_2-v_1\|_{L^2})\|v_2-v_1\|_{L^2}\notag\\
&\quad+|\al_2-\al_1|\|u_1\|_{L^2}\|u_2-u_1\|_{L^2}+|s_2-s_1|\|u_1+v_1\|_{L^2}\|u_2-u_1\|_{L^2},
\end{align}}
where $\de=\frac{\al_2}{\la_k}-1$.
\end{Lemma}
\begin{proof} Consider
\begin{align*}
\langle J^{\prime}_{\al_i, \ba_i}&(u_i+v_i), (u_2- u_1)\rangle\\
&= \int_{Q}((u_i+v_i)(x)-(u_i+v_i)(y))((u_2-u_1)(x)-(u_2-u_1)(y))K(x-y)dxdy\\
&\;-\al_i\int_{\Om}(u_i+v_i)^{+}(u_2-u_1)+\ba_i\int_{\Om}(u_i+v_i)^{-}(u_2-u_1)\\
&= \int_{Q}((u_i+v_i)(x)-(u_i+v_i)(y))((u_2-u_1)(x)-(u_2-u_1)(y))K(x-y)dxdy\\
&\;-\al_i\int_{\Om}(u_i+v_i)(u_2-u_1)+s_i\int_{\Om}(u_i+v_i)^{-}(u_2-u_1).
\end{align*}
Then by using the orthogonality of $ X_1$ and $ X_2$, we obtain
{\small\begin{align*}
\langle J^{\prime}_{\al_i, \ba_i}(u_i+v_i), (u_2- u_1)\rangle&= \int_{Q}((u_i(x)-u_i(y))((u_2-u_1)(x)-(u_2-u_1)(y))K(x-y)dxdy\\
&\quad-\al_i\int_{\Om}u_i(u_2-u_1)+s_i\int_{\Om}(u_i+v_i)^{-}(u_2-u_1).
\end{align*}}
Subtracting the above expression for $i=1,2$ gives
\begin{align*}
&\langle( J^{\prime}_{\al_2, \ba_2}(u_2+v_2)-J^{\prime}_{\al_1, \ba_1}(u_1+v_1)), (u_2- u_1)\rangle\\
&= \int_{Q}|(u_2-u_1)(x)-(u_2-u_1)(y)|^2 K(x-y)dxdy - \int_{\Om}(\al_2 u_2 -\al_1 u_1)(u_2-u_1)\\
&\quad + \int_{\Om}(s_2(u_2+v_2)^{-}-s_1(u_1+v_1)^{-})(u_2-u_1)\\
&=\|u_2-u_1\|^2 - \al_2\int_{\Om}|u_2-u_1|^2 + s_2\int_{\Om}((u_2+v_2)^{-}- (u_1+v_1)^{-})(u_2-u_1)\\
&\quad + (s_2-s_1)\int_{\Om}(u_1+v_1)^{-}(u_2-u_1)- (\al_2-\al_1)\int_{\Om}u_1(u_2-u_1)
\end{align*}
\noi Now we analyze each term of the right hand side separately.
First, it is clear from the definition of $X_1$ and the standard
characterization of the eigenvalue of $(-\De)^s$  that
\[\|u_2-u_1\|^2 - \al_2 \int_{\Om}|u_2-u_1|^2\leq \left(1-\frac{\al_2}{\la_k}\right)\|u_2-u_1\|^2=-\de \|u_2-u_1\|^2.\]
From the H\"{o}lder's inequality, we obtain
\begin{align*}
\int_{\Om}u_1(u_2-u_1)\leq \|u_1\|_{L^2}\|u_2-u_1\|_{L^2}.
\end{align*}
Using the monotonicity of $g(t)= t^-$, the fact that $|g(t_1)- g(t_2)|\leq |t_2-t_1|$ and H\"{o}lder's inequality, we obtain
\begin{align*}
s_2 \int_{\Om}((u_2+v_2)^{-}-&(u_1+v_1)^{-})(u_2-u_1)\\
&= s_2 \int_{\Om}((u_2+v_2)^{-}-(u_1+v_1)^{-})((u_2+v_2)-(u_1+v_1))\\
&\quad- s_2\int_{\Om} ((u_2+v_2)^{-}-(u_1+v_1)^-)(v_2-v_1)\\
&\leq s_2\int_{\Om}|(u_2-u_1)+(v_2-v_1)||v_2-v_1|\\
&\leq s_2 (\|u_2-u_1\|_{L^2}+\|v_2-v_1\|_{L^2})\|v_2-v_1\|_{L^2}.
\end{align*}
Combining the above inequalities together we obtain the desired result.\QED
\end{proof}

\begin{Lemma}
For every $v\in X_2$, the functional $J_{\al,\ba}(\cdot,v): X_1\ra \mb R$ is strictly concave and anticoercive.
\end{Lemma}

\begin{proof} Taking $\al=\al_2=\al_1$, $\ba=\ba_2=\ba_1$ and $v_2=v_1=v$, in \eqref{e1}, we obtain
\[\langle(J^{\prime}_{\al, \ba}(u_2+v)- J^{\prime}_{\al, \ba}(u_1+v)),(u_2 - u_1)\rangle\leq -\de\|u_2-u_1\|^2,\]
which implies strict concavity. The anticoercivity of $J_{\al,\ba}$ now follows from the strict concavity and the Fundamental Theorem of Calculus.\QED
\end{proof}

\begin{Lemma}
For every $v\in X_2$, the functional $J_{\al,\ba}(\cdot,v): X_1
\ra\mb R$ achieves a unique maximum.
\end{Lemma}

\begin{proof}
Let $\{u_k+v\}$ be a maximizing sequence. Then anticoercivity  of
$J_{\al,\ba}$ implies that the sequence $\{u_k\}$ is bounded in
$X_1$. Therefore the sequence $\{u_k\}$ has a weakly convergent
subsequence. Also $J_{\al,\ba}$ is weakly upper semicontinuous,
follows from concavity of $J_{\al,\ba}$. So, $J_{\al,\ba}$ achieves
its maximum. Uniqueness follows easily from the strict
concavity.\QED
\end{proof}

This result makes it possible to define a functional $M_{\al,\ba}: X_2\ra \mb R$ by
\[M_{\al,\ba}(v)=\max_{u\in X_1} J_{\al,\ba}(u,v).\]
Now we investigate a few useful properties of $M_{\al,\ba}$.

\begin{Lemma}
If $t\geq 0$ and $v\in X_2$, then $M_{\al,\ba}(tv)= t M_{\al,\ba}(v)$.
\end{Lemma}

\begin{proof} Case 1:  $t>0$. By the maximizing property of $M_{\al,\ba}$, we have
\begin{align*}
J_{\al,\ba}(M_{\al,\ba}(tv)+tv)\geq J_{\al,\ba}(u+tv)\;\mbox{ for all}\; u\in X_1.
\end{align*}
 Using the homogeneity of $J_{\al,\ba}$, we see that
\begin{align*}
J_{\al,\ba}\left(\frac{M_{\al,\ba}(tv)}{t}+v\right)\geq J_{\al,\ba}\left(\frac{u}{t}+v\right)\;\mbox{for all}\; u\in X_1.
\end{align*}
Hence
\begin{align*}
J_{\al,\ba}\left(\frac{M_{\al,\ba}(tv)}{t}+v\right)\geq J_{\al,\ba}(u+v)\;\mbox{for all}\; u\in X_1.
\end{align*}
Thus for any $t>0$, $M_{\al,\ba}(tv)= t M_{\al,\ba}(v)$.

\noi Case 2:  $t=0$, it only need to argue that $M_{\al,\ba}(0)=0$.
It is immediate that $J_{\al,\ba}(0)=0$. It suffices to show that
$J_{\al,\ba}(u)<0$ for $u\in X_1\setminus \{0\}$. Recall that
$\int_{Q}|u(x)-u(y)|^2 K(x-y)dxdy\leq \la_k\int_{\Om} u^2 dx$ for
all $u\in X_1$ and that $\la_k<\al<\ba$. It follows that
\begin{align*}
J_{\al,\ba}(u)=&\frac{1}{2}\left(\int_{Q}|u(x)-u(y)|^2 K(x-y)- \al \int_{\Om}|u^+|^2 dx -\ba\int_{\Om} |u^-|^2dx\right)\\
\leq& \frac{1}{2}\left(\la_k \int_{\Om} u^2 dx- \al \int_{\Om}|u^+|^2 dx -\ba\int_{\Om} |u^-|^2dx\right)\\
\leq& \frac{1}{2}\left(\la_k \int_{\Om} u^2 dx- \al \int_{\Om}|u^+|^2 dx-\al\int_{\Om} |u^-|^2 dx\right)\\
=&\frac{1}{2}(\la_k -\al)\int_{\Om}|u|^2 dx<0.
\end{align*}
This completes the proof. \QED
\end{proof}


\begin{Lemma}
If $0\ne v \in X_2$, then $M_{\al,\ba}(v)+v$ is sign-changing.
\end{Lemma}

\begin{proof}
Suppose not. Then we assume $w= M_{\al,\ba}(v)+v\gvertneqq 0$ in
$\Om$. Let $n$ represent the Fourier coefficient of $w$ in the
$\phi_1$ direction. We note that $n>0$ because $\int_{\Om} w \phi_1
>0$. Since we have maximized $J_{\al,\ba}$ with respect to $X_1$, we
must have $\langle J^{\prime}_{\al,\ba}(w), \phi_1\rangle =0.$ Thus
\[0= \int_{Q}(w(x)-w(y))(\phi_1(x)-\phi_1(y)) K(x-y)dx dy-\al\int_{\Om}w^{+}\phi_1 dx+\ba\int_{\Om} w^- \phi_1 dx.\]
But $w=w^{+}$ and $w^{-}\equiv 0$, so
\[0= n\int_{Q}|\phi_1(x)-\phi_1(y)|^2K(x-y)dxdy - n \al \int_{\Om} \phi_{1}^{2} dx= n (\la_1-\al)\int_{\Om} \phi_{1}^{2} dx \ne 0,\]
a contradiction. Hence the result. \QED
\end{proof}

In order to obtain the continuity property of $M_{\al,\ba}$, in the next Lemma, we distinguish between the space $X_2$, which has the $X_0$ topology and $Y_2$, which is the set of points in $X_2$ endowed with the $L^2(\Om)$ topology.

\begin{Lemma}
$M_{\al,\ba}$ is locally Lipschitz continuous as a function of $\mb R^2\times Y_2$ into $X_1$.
\end{Lemma}
\begin{proof}
Putting $u_i= M_{\al_i,\ba_i}(v_i)$ for $i=1,2$ into \eqref{e1} to
get
\begin{align*}
\de\|M_{\al_2,\ba_2}(v_2)&-M_{\al_1,\ba_1}(v_1)\|^2\\
&\leq |\ba_2-\al_2|(\|M_{\al_2,\ba_2}(v_2)-M_{\al_1,\ba_1}(v_1)\|_{L^2}+\|v_2-v_1\|_{L^2})\|v_2-v_1\|_{L^2}\\
&\quad+|\al_2-\al_1|\|M_{\al_1,\ba_1}(v_1)\|_{L^2}\|M_{\al_2,\ba_2}(v_2)-M_{\al_1,\ba_1}(v_1)\|_{L^2}\\
&\quad+|s_2-s_1|\|M_{\al_1,\ba_1}(v_1)+v_1\|_{L^2}\|M_{\al_2,\ba_2}(v_2)-M_{\al_1,\ba_1}(v_1)\|_{L^2}.
\end{align*}
\noi By Poincare's inequality, we obtain
\begin{align*}
\de\|M_{\al_2,\ba_2}&(v_2)-M_{\al_1,\ba_1}(v_1)\|^2 \\
&\leq |\ba_2-\al_2|\left(\frac{1}{\la_1}\|M_{\al_2,\ba_2}(v_2)-M_{\al_1,\ba_1}(v_1)\|+\|v_2-v_1\|_{L^2}\right)\|v_2-v_1\|_{L^2}\\
&\quad+|\al_2-\al_1|\|M_{\al_1,\ba_1}(v_1)\|_{L^2}\frac{1}{\la_1}\|M_{\al_2,\ba_2}(v_2)-M_{\al_1,\ba_1}(v_1)\|\\
&\quad+|s_2-s_1|\|M_{\al_1,\ba_1}(v_1)+v_1\|_{L^2}\frac{1}{\la_1}\|M_{\al_2,\ba_2}(v_2)-M_{\al_1,\ba_1}(v_1)\|.
\end{align*}
Taking $v_2=v$, $v_1=0$, $\al_1=\al_2=\al$ and $\ba_1=\ba_1=\ba$. Note that $M_{\al,\ba}(0)=0$. Then the above inequality reduces to
\begin{align*}
\de\|M_{\al,\ba}(v)\|^2 &\leq
|\ba-\al|\left(\frac{1}{\la_1}\|M_{\al,\ba}(v)\|+\|v\|_{L^2}\right)\|v\|_{L^2}.
\end{align*}
From this inequality, one can show that $\|M_{\al,\ba}(v)\|\leq C\|v\|_{L^2}$ for an appropriate $C>0$ depending on $\de$.

\noi We now proceed to the main estimate. For a given $v_1$, we let $c_1=\|M_{\al_1,\ba_1}(v_1)\|_{L^2}$, $c_2=\|M_{\al_1,\ba_1}(v_1)+v_1\|_{L^2}$ and $w=\|M_{\al_2,\ba_2}(v_2)-M_{\al_1,\ba_1}(v_1)\|$.
It follows that
\[\de w^2\leq (|\ba_2-\al_2|\|v_2-v_1\|_{L^2}+ c_1|\al_2-\al_1|+c_2|s_2-s_1|)\frac{1}{\la_1}w +|\ba_2-\al_2|\|v_2-v_1\|_{L^2}^2.\]
Let $\gamma:= (|\ba_2-\al_2|\|v_2-v_1\|_{L^2}+
c_1|\al_2-\al_1|+c_2|s_2-s_1|)$ and observe that
$|\ba_2-\al_2|\|v_2-v_1\|_{L^2}\leq \gamma,$ so
\[\de w^2 \leq \frac{\gamma}{\la_1} w +
\frac{\gamma^2}{|\ba_2-\al_2|}.\]
It follows that there is a positive constant $K$ such that $z\leq K\gamma$ and the result follows.\QED
\end{proof}

\begin{Lemma}
For a given $\al$ and $\ba$, $M_{\al,\ba}: Y_2\ra X_1$ is globally Lipschitz continuous.
\end{Lemma}

\begin{proof}
Taking $\al_1=\al_2=\al$ and $\ba_1=\ba_2=\ba$ in the previous proof, we can easily seen that $w\leq K_1 \gamma$, where $\gamma= \|v_2-v_1\|$, and $K_1$ has no dependence on $c_1$ and $c_2$.\QED
\end{proof}
\begin{Lemma}
There is a $\rho>0$ such that $\|M_{\al,\ba}(v)\|\leq \rho \|v\|_{L^2}$ for all $v\in X_2$.
\end{Lemma}

\begin{proof}
It follows from the globally Lipschitz continuity and homogeneity properties of $M_{\al,\ba}$.\QED
\end{proof}

\begin{Lemma}
Suppose that $\{v_k\}$ is bounded in $X_2$, and $\{\al_k\}$, $\{\ba_k\}$
are bounded sequences in $\mb R$ that satisfy our given
restriction on $(\al,\ba)$. Then there exist subsequences, still
denoted by $\{v_k\}$, $\{\al_k\}$ and $\{\ba_k\}$ such that
$(\al_k,\ba_k)\ra (\al,\ba)$ in $\mb R^2$, $v_k\rightharpoonup v$ in
$X_2$, $v_k\ra v$ in $L^{2}(\Om)$ and $M_{\al,\ba}(v_k)\ra
M_{\al,\ba}(v)$ in $X_1$.
\end{Lemma}

\begin{proof}
The proof follows from the standard compactness arguments combined
with the continuity established in the previous Lemma.\QED
\end{proof}
\begin{Lemma}
If $J_{\al,\ba}$ has a critical point at $w= u+v$, then $u= M_{\al,\ba}(v)$.
\end{Lemma}

\begin{proof}
It is a straight forward consequence of strict concavity.\QED
\end{proof}

\noi Given the last Lemma it makes sense to restrict our search for critical points to the set $\mc X_2:=\{M_{\al,\ba}(v)+v: v\in X_2\}$.

We define
\[\tilde{J}_{\al,\ba}: X_2\ra \mb R\;\mbox{as}\; \tilde{J}_{\al,\ba}(v)=J_{\al,\ba}(M_{\al,\ba}(v)+v).\]
\begin{Lemma}
The functional $\tilde{J}_{\al,\ba}$ is continuously differentiable.
\end{Lemma}
\begin{proof}
Using the maximum property and the continuity of $M_{\al,\ba}$, as well as the fact that $J_{\al,\ba}$ is $C^{1}$ on $X_0$, we have the following inequality
\begin{align*}
\tilde{J}_{\al,\ba}(v_2)-\tilde{J}_{\al,\ba}(v_1)&=J_{\al,\ba}(M_{\al,\ba}(v_2)+v_2)- J_{\al,\ba}(M_{\al,\ba}(v_1)+v_1)\\
&\leq J_{\al,\ba}(M_{\al,\ba}(v_2)+v_2)-J_{\al,\ba}(M_{\al,\ba}(v_2)+v_1)\\
&= \langle J^{\prime}_{\al,\ba}(M_{\al,\ba}(v_2)+v_1),(v_2-v_1)\rangle + o(\|v_2-v_1\|)\\
&=\langle J^{\prime}_{\al,\ba}(M_{\al,\ba}(v_1)+v_1),(v_2-v_1)\rangle+o(\|v_2-v_1\|)\\
&\quad+\langle(J^{\prime}_{\al,\ba}(M_{\al,\ba}(v_2)+v_1)- J^{\prime}_{\al,\ba}(M_{\al,\ba}(v_1)+v_1)),(v_2-v_1)\rangle\\
& =\langle J^{\prime}_{\al,\ba}(M_{\al,\ba}(v_1)+v_1),(v_2-v_1)\rangle+o(\|v_2-v_1\|).
\end{align*}
Similarly we can show
\begin{align*}
\tilde{J}_{\al,\ba}(v_2)-\tilde{J}_{\al,\ba}(v_1)\geq\langle
J^{\prime}_{\al,\ba}(M_{\al,\ba}(v_1)+v_1),(v_2-v_1)\rangle+
o(\|v_2-v_1\|).
\end{align*}
Hence the result follows.\QED

\end{proof}
From the above Lemma, we also note the following identity
\begin{align}\label{e2}
\tilde{J}^{\prime}_{\al,\ba}(v)=\tilde{J}^{\prime}_{\al,\ba}(M_{\al,\ba}(v)+v).
\end{align}

\begin{Lemma}
$v\in X_2$ is a critical point of $\tilde{J}_{\al,\ba}$ if and only if $M_{\al,\ba}(v)+v$ is a critical point of $J_{\al,\ba}$.
\end{Lemma}
\begin{proof}
Assume that $M_{\al,\ba}(v)+v$ is a critical point of $J_{\al,\ba}$.
Then $\langle J^{\prime}_{\al,\ba}(M_{\al,\ba}(v)+v), w\rangle=0$
for all $w\in X_0$. In particular
$\langle J^{\prime}_{\al,\ba}(M_{\al,\ba}(v)+v), w\rangle=0$ for all $w\in X_2$. Using the equation \eqref{e2}, we have $\langle \tilde{J}^{\prime}_{\al,\ba}(v), w\rangle=0$
 for all $w\in X_2$, so $v$ is a critical point of $\tilde{J}_{\al,\ba}$.\\
Conversely, suppose that $v$ is a critical point of
$\tilde{J}_{\al,\ba}$. Then as above, $\langle
J^{\prime}_{\al,\ba}(M_{\al,\ba}(v)+v), w\rangle=0$ for all $w\in
X_2$. Recall that $M_{\al,\ba}(v)$ maximizes $J_{\al,\ba}(u+v)$ for
$u\in X_1$. Hence $\langle J^{\prime}_{\al,\ba}(M_{\al,\ba}(v)+v),
u\rangle =0$ for all $u\in X_1$. Thus $\langle
J^{\prime}_{\al,\ba}(M_{\al,\ba}(v)+v), w\rangle=0$ for all $w\in
X_0$.\QED
\end{proof}
\begin{Lemma}
$\tilde{J}_{\al,\ba}(tv)= t^2 \tilde{J}_{\al,\ba}(v)$ for all $t\geq 0$ and for all $v\in X_2$.
\end{Lemma}
\begin{proof}
Using the homogeneity of $J_{\al,\ba}$ and $M_{\al,\ba}$, we have
\begin{align*}
\tilde{J}_{\al,\ba}(tv)= & J_{\al,\ba}(M_{\al,\ba}(tv)+tv)= J_{\al,\ba}(tM_{\al,\ba}(v)+tv)\\
= & t^2 J_{\al,\ba}(M_{\al,\ba}(v)+v)= t^2 \tilde{J}_{\al,\ba}(v).
\end{align*}\QED
\end{proof}
The given homogeneity leads to the following Lemma:
\begin{Lemma}\label{le14}
If $v\in X_2$ is a critical point of $\tilde{J}_{\al,\ba}$ then
$\tilde{J}_{\al,\ba}(v)=0$.
\end{Lemma}

\begin{proof}
Differentiate $\tilde{J}_{\al,\ba}(tv)= t^2 \tilde{J}_{\al,\ba}(v)$
with respect to $t$ to get $\langle \tilde{J}^{\prime}_{\al,\ba}(tv),
v\rangle= 2t \tilde{J}_{\al,\ba}(v)$. Let $t=1$ and the result
follows.\QED
\end{proof}

As with $J_{\al,\ba}$, it is helpful to think of $\tilde{J}_{\al,\ba}$ as a function on $\mb R^2\times X_2$ as $\tilde{J}_{\al,\ba}(v):=\tilde{J}(\al,\ba,v)$. Then we establish the following:
\begin{Lemma}
$\tilde{J}(\al,\ba,\cdot):= \tilde{J}_{\al,\ba}(\cdot)$ is strictly decreasing in $\al$ and $\ba$.
\end{Lemma}
\begin{proof} Assume that $\al_1 \leq \al_2$ and $\ba_1\leq \ba_2$, where at least one of the inequality is strict. Then using the definition of $J_{\al,\ba}$, the fact that $M_{\al,\ba}(v)+v$ is sign changing and the maximizing property of $M_{\al,\ba}$, we obtain
\begin{align*}
\tilde{J}(\al_2,\ba_2,v)= & J(\al_2,\ba_2, M(\al_2,\ba_2,v)+v)\\
< &J(\al_1,\ba_1, M(\al_2,\ba_2,v)+v)\\
\leq&J(\al_1,\ba_1, M(\al_1,\ba_1,v)+v),
\end{align*}
which completes the proof.\QED
\end{proof}

\begin{Lemma}
Given any positive number $R$, there is a positive number $C$ such that
\begin{align*}
|\tilde{J}(\al_2,\ba_2,v)-\tilde{J}(\al_1,\ba_1,v)|\leq C (|\al_2-\al_1|+|\ba_2-\ba_1|),
\end{align*}
whenever $\max\{|\al_1|,|\al_2|,|\ba_1|,|\ba_2|,\|v\|\}\leq R$.
\end{Lemma}
\begin{proof}
Combining the Lipschitz continuity of $J_{\al,\ba}$ and $M_{\al,\ba}$, we obtain the desire result. We also notice that the bound is on $\|v\|$ rather than just $\|v\|_{L^2}$. This is because the Lipschitz constant on $J_{\al,\ba}$ depends on a bound in $X_0$.\QED
\end{proof}
\subsection{Minimizing in the $X_2$ direction}
\noi We note that the search for critical points of $J_{\al,\ba}$ on
$X_0$ has been reduced to a search for critical points of
$\tilde{J}_{\al,\ba}$ on $X_2$. We know that $\tilde{J}_{\al,\ba}$
is homogeneous, so it suffices to look for critical points on $\mc
S_{X_2}:= \{v\in X_2: \|v\|_{L^2}=1\}$, a weakly closed set in $X_0$
i.e. if $\{v_k\}\subset \mc S_{X_{2}}$ and $v_k \rightharpoonup v$
weakly in $X_0$, then $v_k\ra v$ strongly in $L^2$ so $\|v\|_{L^2}=1
$ and $v\in \mc S_{X_2}$.
\begin{Lemma}\label{le5} $\tilde{J}_{\al,\ba}$ achieves a
global minimum on $\mc S_{X_2}$.
\end{Lemma}
\begin{proof}
It is easy to see that $\tilde{J}_{\al,\ba}$ is bounded below on
$\mc S_{X_2}$.  Let $\{v_k\}\subset \mc S_{ X_2}$ be a minimizing
sequence for $\tilde{J}_{\al,\ba}$ and let $m= \inf_{v\in \mc
S_{X_2}}\tilde{J}_{\al,\ba}(v)$. It is easy to see that $\|v_k\|$ is
bounded, so without loss of generality, $v_k \rightharpoonup v_0$
weakly in $X_0$ and $v_k\ra v_0$ strongly in $L^2(\Om)$ with
$\|v_0\|_{L^2}=1$. By the continuity and compactness of
$M_{\al,\ba}$, we have $M_{\al, \ba}(v_k)\ra M_{\al,\ba}(v_0)$ in
$X_{0}$. Using these observation as well as the weak lower
semicontinuity of $X_0$ norm we obtain that $v_0\in \mc S_{X_2}$
such that $\tilde{J}_{\al,\ba}(v_0)= \inf_{v\in \mc
S_{X_2}}\tilde{J}_{\al,\ba}(v)$.\QED
\end{proof}

If $v_0$ is a critical point of $\tilde{J}_{\al,\ba}$ restricted to
$\mc S_{X_2}$, then we can not conclude that it is a critical point
of $\tilde{J}_{\al,\ba}$ on $X_2$. We must check the direction
orthogonal to the surface $\mc S_{X_2}$.
\begin{Lemma}
$v_0 \in X_2$ is a nontrivial critical point of $\tilde{J}_{\al,\ba}$ if and only if $v_0$ is a critical point of $\tilde{J}_{\al,\ba}$ restricted to $\mc S_{X_2}$ and $\tilde{J}_{\al,\ba}(v_0)=0$.
\end{Lemma}
\begin{proof}
This is a standard fact for homogeneous operator, since every
nontrivial element of $X_2$ can be written as $tv$ for some $v\in
S_{X_2}$ and for some $t>0$. Computing derivatives separately with
respect to $t$ and $v$ gives the result. One can see the proof of
Lemma \ref{le14}.\QED

\end{proof}

\begin{Lemma}
If $u$ is a nontrivial critical point of ${J}_{\al,\ba}$ if and only if $u= M_{\al,\ba}(v_0)+v_0$, where $\frac{v_0}{\|v_0\|_{L^2}}$ is a critical point of $\tilde{J}_{\al,\ba}$ restricted to $\mc S_{X_2}$ and $\tilde{J}_{\al,\ba}(v_0)=0$.
\end{Lemma}
\begin{proof}
It is a direct consequence of the previous Lemma. \QED
\end{proof}

\noi Now we define
\[m(\al,\ba):= \min_{v\in \mc S_{X_2}}\tilde{J}_{\al,\ba}(v). \]
\begin{Lemma}
$m(\al,\ba)$ is Lipschitz continuous and is strictly decreasing as a function of both $\al$ and $\ba$. Moreover, $m(\al,\al)>0$
\end{Lemma}
\begin{proof}
Let $(\al_1,\ba_1)$ and $(\al_2, \ba_2)$ be two points in the plane.
Let $v_1$ and $v_2$ be the corresponding global minimizers on $\mc
S_{X_2}$, and let $w_{ij}= M_{\al_i,\ba_i}(v_j)+v_j$ for $i,j=1,2$.
Then using the minimizing property of $v_i$ and then the maximizing
property of $M_{\al_j,\ba_j}$, we obtain
\begin{align*}
m(\al_i,\ba_i)= & J_{\al_i,\ba_i}( M_{\al_i,\ba_i}(v_i)+v_i)\\
\leq & J_{\al_i,\ba_i}( M_{\al_i,\ba_i}(v_j)+v_j)\\
=&J_{\al_j,\ba_j}( M_{\al_i,\ba_i}(v_j)+v_j) +\frac{1}{2} (\al_j-\al_i)\int_{\Om}(w^{+}_{ij})^2+\frac{1}{2} (\ba_j-\ba_i)\int_{\Om}(w^{-}_{ij})^2\\
\leq& J_{\al_j,\ba_j}( M_{\al_j,\ba_j}(v_j)+v_j) +\frac{1}{2} (\al_j-\al_i)\int_{\Om}(w^{+}_{ij})^2+\frac{1}{2} (\ba_j-\ba_i)\int_{\Om}(w^{-}_{ij})^2\\
=& m(\al_j,\ba_j)+\frac{1}{2}
(\al_j-\al_i)\int_{\Om}(w^{+}_{ij})^2+\frac{1}{2}
(\ba_j-\ba_i)\int_{\Om}(w^{-}_{ij})^2.
\end{align*}
Since this inequality holds for $i=1$, $j=2$ and $i=2$, $j=1$, we have
\begin{align}
|m(\al_2,\ba_2)-m(\al_1,\ba_1)|\leq c (|\al_2-\al_1|+|\ba_2-\ba_1|),
\end{align}
where $c= \max\{\|w_{12}\|_{L^2},\|w_{21}\|_{L^2}\}$. Moreover, if $\al_2\geq \al_1$ and $\ba_2\geq \ba_1$ where at least one of these inequalities is strict, then $m(\al_2,\ba_2)<m(\al_1,\ba_1)$. This last conclusion uses the fact that $w_{ij}$ is sign-changing.

If $\al=\ba$, then for $w\in X_0$ we have
\begin{align*}
J_{\al,\ba}(w)&= J_{\al,\al}(w)\\
&= \frac{1}{2}\left(\int_{Q}|w(x)-w(y)|^2 K(x-y)dxdy  -\al\int_{\Om} w^2 dx \right)\\
&= \frac{1}{2} \sum_{j=1}^{\infty} (\la_j-\al)c^{2}_{j},
\end{align*}
where we are applying the Fourier decomposition of $w$. We will write $w=u+v$ using the usual decomposition. The coefficient $(\la_i - \al)$ are strictly negative for $j\leq k$, so it follows that we can maximize in the $X_1$ direction by choosing $c_j=0$ for $j=1,\cdots,k$, i.e. $M_{\al,\ba}(v)\equiv 0$. Thus we have
\[\tilde{J}_{\al,\ba}(v)= J_{\al,\ba}(v)= \frac{1}{2} \sum_{j=k+1}^{\infty} (\la_j- \al) c_{j}^{2}.\]
The coefficients $(\la_j-\al)$ are strictly positive and increasing
for $j=k+1,k+2,\cdots$. Also $\sum_{j=k+1}^{\infty}
c_{j}^{2}=\|v\|_{L^2}=1$. Using the Lagrange multipliers one can
show that the critical points of this sum occur when $c_j\equiv \pm
1$ for one $j$ and $c_j=0$ for all other $j$. The minimizing choice
is when $c_{k+1}=1$ and $c_j=0$ for $j>k+1$. Hence the minimizer is
$v= \pm\phi_{k+1}$ and $m(\al,\al)= \tilde{J}_{\al,\ba}(v)=
\frac{1}{2}(\la_{k+1}-\al)>0.$\QED
\end{proof}

\begin{Lemma}
$m(\al,\la_{k+1})>0$.
\end{Lemma}
\begin{proof}
Let $v\in \mc S_{X_2}$ and let $\ba=\la_{k+1}$.
\begin{align*}
\tilde{J}_{\al,\ba}(v)&=J_{\al,\ba}(M_{\al,\ba}(v)+(v))\\
&\geq J_{\al,\ba}(v)\\
&= \frac{1}{2} \left(\int_{Q}|v(x)-v(y)|^2 K(x-y) dxdy -\al\int_{\Om } (v^{+})^2 dx -\la_{k+1}\int_{\Om } (v^{-})^2 dx\right)\\
&> \frac{1}{2} \left(\int_{Q}|v(x)-v(y)|^2 K(x-y) dxdy -\la_{k+1}\int_{\Om } v^2 dx\right)\\
&\geq 0,
\end{align*}
where we have used the fact that $\al<\la_{k+1}$ and $v^{+}$ is nontrivial.\QED
\end{proof}
All of the Lemmas above have been leading to

\begin{Theorem}
Assume that $\la_k< \al<\la_{k+1}$. Then one of the following is true:
\begin{enumerate}
\item[1.] $m(\al,\ba)>0$ and $(\al,\ba)\not\in \sum_K$ for all $\ba\geq \al$.
\item[2.] There is a unique $\ba(\al)>\la_{k+1}$, such that $m(\al,\ba(\al))=0$. Moreover, \\$(\al,\ba(\al)) \in \sum_K$, but $(\al,\ba)\not\in \sum_K$ if $\al\leq\ba<\ba(\al)$.
\end{enumerate}
\end{Theorem}

\begin{Lemma}
The curve $(\al,\ba(\al))$ is Lipschitz continuous, strictly decreasing, and contains the point $(\la_{k+1},\la_{k+1})$.
\end{Lemma}

\begin{proof}
Consider two points $(\al_1,\ba_1)$ and $(\al_2,\ba_2)$ on $\sum_K$, characterized as above, with $\al_2>\al_1$. Let $v_i$ be a minimizer of $J_{\al_i,\ba_i}(M_{\al_i,\ba_i}(v)+v)$ such that $\|v_i\|_{L^2}=1$. In particular, we know that $J_{\al_i,\ba_i}(M_{\al_i,\ba_i}(v_i)+v_i)=0$ and that $J_{\al_i,\ba_i}(M_{\al_i,\ba_i}(v)+v)\geq 0$ for all $v\in X_2$. Let $w_i= M_{\al_i,\ba_i}(v_i)+v_i$, then we have
\begin{align*}
0&= 2 J_{\al_1,\ba_1}(w_1)\\
&= \int_{Q}|w_1(x)- w_1(y)|^2 K(x-y) dxdy - \al_1 \int_{\Om} (w_{1}^{+})^{2} dx- \ba_1 \int_{\Om} (w_{1}^{-})^{2}dx\\
&> \int_{Q}|w_1(x)- w_1(y)|^2 K(x-y) dxdy - \al_2 \int_{\Om}
(w_{1}^{+})^{2}dx - \ba_1 \int_{\Om} (w_{1}^{-})^{2}dx,
\end{align*}
where we obtain strict inequality using the fact that $\al_2>\al_1$ and that $w_1$ is sign changing so that $w^{+}_{1}$ is nontrivial. It follows that $m(\al_2,\ba_1)<0$. Since $m(\al,\ba)$ is strictly decreasing in $\ba$ and $m(\al_2,\ba_2)=0$, it must be the case that $\ba_2<\ba_1$, i.e. $\ba(\al)$ is strictly decreasing.
Now consider
\begin{align*}
2 J_{\al_2,\ba_1}(w_2)&= \int_{Q}|w_2(x)- w_2(y)|^2 K(x-y) dxdy - \al_2 \int_{\Om} (w_{2}^{+})^{2} - \ba_1 \int_{\Om} (w_{2}^{-})^{2}\\
&= (\ba_2- \ba_1) \int_{\Om} (w_{2}^{-})^{2}.
\end{align*}
It follows that $m(\al_2,\ba_1)\leq \frac{1}{2}(\ba_2- \ba_1)
\int_{\Om} (w_{2}^{-})^{2} <0$. Thus
\[|\ba_2-\ba_1| \leq 2 \frac{1}{\int_{\Om} (w_{2}^{-})^2}|m(\al_2,\ba_1)|= 2\frac{1}{\int_{\Om} (w_{2}^{-})^2}|m(\al_2,\ba_1)-m(\al_2,\ba_2)|.\]
The Lipschitz estimate for $\ba(\al)$ now follows from the Lipschitz estimate for $m(\al,\ba)$.\QED
\end{proof}

\section{Nonresonance and resonance case for problem $(P_{\la})$}
\subsection{The Nonresonance Case:}

In this section we assume that $(\al,\ba)\in \mb R^2$ such that $\la_k<\al<\la_{k+1}$ and $\al\leq \ba<\ba(\al)$. By the characterization of the Fu\v{c}ik spectrum above, we know that $(\al, \ba)\not\in \sum_K$. By analogy with the Fredholm Alternative for the linear case, we should expect that $(P_{\la})$ is solvable without further restrictions on either $f$ or $h$, and this is indeed the case.

For notational convenience let $E=E_{\al,\ba}$ and $J=J_{\al,\ba}$. Notice that
\[E(u) = J(u)-\int_{\Om} (F(u)+hu).\]

We will see that the geometry of $J$ dominates the geometry of $E$, so that the saddle geometry is easily proved in this case.

\begin{Lemma}
There is a positive constant $K$ such that $|\int_{\Om}(F(u)+hu)|\leq K \|u\|_{L^2}$ for all $u\in X_0$.
\end{Lemma}
\begin{proof}
Since $f$ is bounded, there is an $M>0$ such that $|f(t)|\leq M$ for all $t\in \mb R$. It immediately follows that $|F(t)|\leq M|t|$ for all $t$. Thus
\[\left|\int_{\Om} (F(u)+hu)\right| \leq  \int_{\Om} |F(u)+hu| \leq \int_{\Om} (M+|h|)|u|\leq \left(\int_{\Om} (M+|h|)^2\right)^{\frac12}\left(\int_{\Om} u^2\right)^{\frac{1}{2}}.\]\QED
\end{proof}

\begin{Lemma}
$E$ is anticoercive when restricted to $X_1$.
\end{Lemma}
\begin{proof}
Let $u\in X_1$, then using $\al\leq\ba$ and $\int_{Q} |u(x)-u(y)|^2
K(x-y)dxdy \leq \la_k\int_{\Om}u^2 dx$ for all $u\in X_1$ we have
\begin{align*}
E(u)=& J(u)-\int_{\Om} (F(u)+hu)\\
\leq & \left(1-\frac{\al}{\la_k}\right)\|u\|^2 + \left(\int_{\Om} (M+|h|)^2\right)^{\frac12}\left(\int_{\Om} u^2\right)^{\frac{1}{2}}\\
\leq & \left(1-\frac{\al}{\la_k}\right)\|u\|^2 + C \|u\| \ra
-\infty,
\end{align*}
as $\|u\|\ra\infty$, since $\la_k<\al$.\QED
\end{proof}

\begin{Lemma}\label{le9}
The functional $E$ is bounded below on $\mc X_2:= \{M_{\al,\ba}(v)+v : v\in X_2\}.$
\end{Lemma}
\begin{proof}
Since $\ba<\ba(\al)$, we know that $\inf_{\mc S_{X_2}}\tilde{J}(v)\geq c$ for some $c$. It follows that for any $v\in \mc X_2$
\[J(M(v)+v)=\tilde{J}(v)=\|v\|_{L^2}^2 \tilde{J}\left(\frac{v}{\|v\|_{L^2}}\right)\geq c\|v\|_{L^2}^2.\]
Thus for $u=M(v)+v$
\[E(u)\geq c \|v\|_{L^2}^{2}- \left(\int_{\Om} (M+|h|)^2\right)^{\frac12}\|u\|_{L^2}. \]
Recall that $\|M_{\al,\ba}(v)\|\leq \rho\|v\|_{L^2}$ for all $v\in X_2$. It follows that $\|u\|_{L^2}\leq k\|v\|_{L^2}$ for some $k>0$ and all $v\in X_2$.
Thus our inequality for $E$ becomes
\[E(u)\geq c\|v\|_{L^2}^{2} - k \left(\int_{\Om} (M+|h|)^2\right)^{\frac12}\|v\|_{L^2}.\]
We conclude that $E$ is bounded below, and is even coercive on $\mc X_2$.\QED
\end{proof}

\noi As a result of this estimates above we can choose $R>0$ such that
\[\sup_{u\in  X_1: \|u\|=R} E(u)< \inf_{v\in \mc X_2} E(v).\]

\noi In the next Lemma we show that $\partial B_{R}(0):= \{x\in X_1:
\|x\|=R\}$ and $\mc X_2$ link. Note that $\partial B_R(0)$ is
clearly embedding of $S^{k-1}$ in $X_0$.
\begin{Lemma}\label{le10}
Let $\ga: \overline{B_{R}(0)}\subset X_1 \ra X_0$ be a continuous
function and write $\ga(x)=\ga_{X_1}(x)+\ga_{X_2}(x)$, where
$\ga_{X_1}(x)\in X_1$ and $\ga_{X_2}(x)\in X_2$. We assume that
$\ga$ fixes $\partial B_R$, so $\ga_{X_1}(x)=x$ and $\ga_{X_2}(x)=0$
for all $x\in\partial\overline{B_{R}(0)},$ then $\ga(
\overline{B_{R}(0)})\cap \mc X_2\ne\emptyset$.
\end{Lemma}
\begin{proof}
We must show that there is an $x\in \overline{B_{R}(0)}$ such that $\gamma_{X_1}(x)= M(\gamma_{X_2}(x))$, so it is reasonable to study the solutions of the equation $G(x)=0$
where $G:\overline{B_{R}(0)}\ra X_1$: $G(x)= \ga_{X_1}(x)- M(\ga_{X_2}(x))$. It is clear that $G$ is continuous. Also, if $x\in \partial \overline{B_{R}(0)}$, then $G(x)=x\ne 0$ and so the Brouwer degree $deg(G,\overline{B_{R}(0)},0)$ is well defined. Consider the homotopy $h(t,x)= tG(x)+(1-t)x$, where $t\in[0,1]$ and $x\in \overline{B_{R}(0)}$. For $x\in\partial\overline{B_{R}(0)}$ we have $h(t,x)= tx +(1-t)x =x\ne0 $, so $deg(G,\overline{B_{R}(0)} ,0)= deg(I,\overline{B_{R}(0)},0)=1$ where $I$ represents the identity map. Hence $G(x)=0$ has a solution in $\overline{B_{R}(0)}$.\QED
\end{proof}

\begin{Lemma}\label{le11}
Let $K: \mb R^n\setminus \{0\}\ra (0,\infty)$ satisfy assumptions
$(K1)-(K3)$ and let $f:\mb R\ra \mb R$ is bounded and continuous,
and $h\in L^2(\Om)$. Let $c\in \mb R$ and let $\{u_k\}_{k\in \mb N}$
be a sequence in $X_0$ such that
\begin{align}\label{p1}
E(u_k)\leq c,
\end{align}
and
\begin{align}\label{p2}
\sup\{|\langle E^{\prime}(u_k),\phi\rangle : \phi\in X_0, \|\phi\| =1|\}\ra 0
\end{align}
as $k\ra\infty$. Then, the sequence $\{u_k\}_{k\in\mb N}$ is bounded in $X_0$.
\end{Lemma}

\begin{proof}
Let $\{u_k\}\subset X_0$ such that \eqref{p1} and \eqref{p2} holds
i.e. $E(u_k)$ is bounded and $E^{\prime}(u_k)\ra 0$ in $X_{0}^{*}$.
Then we show that $\{u_k\}$ is bounded in $X_0$. Suppose that
$\|u_k\|_{L^2}$ is unbounded. Without loss of generality we may
assume that $\|u_k\|_{L^2}$ is increasing to $\infty$. Consider
$v_k:=\frac{u_k}{\|u_k\|_{L^2}}$. Now,
\begin{align*}
\frac{E(u_k)}{\|u_k\|^{2}_{L^2}}&=\frac{1}{2} \int_{Q}|v_k(x)-v_k(y)|^2 K(x-y) dxdy - \frac{\al}{2}\int_{\Om} (v_{k}^{+})^2-\frac{\ba}{2}\int_{\Om} (v_{k}^{-})^2\\
 &\quad  -\frac{1}{\|u_k\|_{L^2}^{2}} \int_{\Om} (F(u_k)+ h u_k).
\end{align*}
Then using \eqref{p1}, it is clear that
$\frac{E(u_k)}{\|u_k\|^{2}_{L^2}}\ra 0$. Also $
\frac{\al}{2}\int_{\Om} (v_{k}^{+})^2 + \frac{\ba}{2}\int_{\Om}
(v_{k}^{-})^2 + \frac{1}{\|u_k\|_{L^2}^{2}}\int_{\Om} (F(u_k)+ h
u_k)$ is bounded. Thus $\{v_k\}$ is bounded in $X_0$ and $X_0$ is a
reflexive space (being a Hilbert space), so
 up to a subsequence, there exists $v_0\in X_0$ such that $v_k \rightharpoonup v_0$ weakly in $X_0$, $v_k\ra v_0$ strongly in $L^{2}(\Om)$ and $\|v_0\|_{L^2}=1$.\\
\noi Now for any $w\in X_0$, we consider
\begin{align*}
\left\langle \frac{E^{\prime}(u_k)}{\|u_k\|_{L^2}}, w \right\rangle&= \int_{Q}(v_k(x)-v_k(y))(w(x)-w(y)) K(x-y) dxdy- \al\int_{\Om} (v_{k}^{+})w\\
 &\quad  +\ba\int_{\Om} (v_{k}^{-})w -\frac{1}{\|u_k\|_{L^2}} \left(\int_{\Om} (f(u_k)+ h)w\right).
\end{align*}
Using the boundedness of $f$ it is clear that
$\frac{1}{\|u_k\|_{L^2}} \int_{\Om} (f(u_k)+ h)w\ra 0$. Also using
the $L^2$ convergence of $v_k$, it is clear that $v_{k}^{+}$ and
$v_{k}^{-}$ converges to $v_{0}^{+}$ and $v_{0}^{-}$ respectively in
$L^{2}$. So,
\[- \al\int_{\Om} (v_{k}^{+})w + \ba\int_{\Om} (v_{k}^{-})w \ra -\al\int_{\Om} (v_{0}^{+})w + \ba\int_{\Om} (v_{0}^{-})w.\]
By the weak convergence of $v_k$ in $X_0$, we have for every $\phi\in X_0$,
{\small\[\int_{Q}  (v_k(x)-v_k(y))(\phi(x)-\phi(y)) K(x-y) dx dy \ra \int_{Q}  (v_0(x) - v_0(y))(\phi(x)-\phi(y)) K(x-y) dx dy.\]}
as $k\ra\infty$. Thus using the above discussion, we obtain
$\left\langle\frac{E^{\prime}(u_k)}{\|u_k\|_{L^2}},
w\right\rangle\ra 0$. Hence
\[0 = \int_{Q} (v_0(x)-v_0(y))(w(x)-w(y))dx dy - \al\int_{\Om} (v_{0}^{+})w + \ba\int_{\Om} (v_{0}^{-})w \;\fa\; w\in X_0. \]
Therefore $v_0$ is a nontrivial weak solution of \eqref{eq01}. This contradicts the fact that $(\al,\ba)\not\in \sum_K$. Hence $\{u_k\}$ is bounded in $L^2$.
Now
\[E(u_k)=\frac{1}{2} \|u_k\|^2 - \frac{\al}{2}\int_{\Om} (u_{k}^{+})^2 -\frac{\ba}{2}\int_{\Om} (u_{k}^{-})^2 -\int_{\Om} (F(u_k)+ h u_k).\]
We see that $E(u_k)$, $\int_{\Om} (u_{k}^{+})^2$, $\int_{\Om} (u_{k}^{-})^2$ and $\int_{\Om} (F(u_k)+ h u_k)$ are all bounded, so $\|u_k\|$ must be bounded.\QED
\end{proof}

\begin{Lemma}\label{le12}
Let $K: \mb R^n\setminus \{0\}\ra (0,\infty)$ satisfy assumptions
$(K1)-(K3)$ and let $f:\mb R\ra \mb R$ is bounded and continuous,
and $h\in L^2(\Om)$. Let $\{u_k\}_{k\in \mb N}$ be a bounded
sequence in $X_0$ such that \eqref{p1} and \eqref{p2} hold true.
Then there exists $u_{0}\in X_0$ such that, up to a subsequence
\begin{align*}
\|u_k- u_0\|\ra 0 \;\mbox{as}\; k\ra\infty.
\end{align*}
\end{Lemma}
\begin{proof}
Let $\{u_k\}_{k\in \mb N}$ be a bounded sequence in $X_0$. Then up to a subsequence, there exists $u_0\in X_0$ such that $u_k$ converges to $u_0$ weakly in $X_0$, i.e. for every $\phi\in X_0$
{\small\[\int_{Q}(u_k(x)-u_k(y))(\phi(x)-\phi(y)) K(x-y) dx dy \ra \int_{Q}  (u_0(x)-u_0(y))(\phi(x)-\phi(y)) K(x-y) dx dy,\]}
as $k\ra\infty$. Moreover, up to a subsequence $u_k\ra u_0$ strongly
in $L^{\mu}(\Om)$ for any $\mu \in [1, 2^*_s)$ and $u_k(x)\ra
u_0(x)$ a.e. in $\mb R^n$ as $k\ra \infty$. Now,
\begin{align}\label{p3}
\langle E^{\prime}&(u_k), (u_k-u_0)\rangle =  \int_{Q}(u_k(x)-u_k(y))((u_k-u_0)(x)-(u_k-u_0)(y))K(x-y)dxdy \notag\\
&-\al\int_{\Om} (u_{k}^{+})(u_k-u_0) + \ba\int_{\Om}
(u_{k}^{-})(u_k-u_0) - \int_{\Om} (f(u_{k})+h)(u_k-u_0).
\end{align}
Also using the $L^2$ boundedness of $u_{k}^{+}$, $u_{k}^{-}$, and $f(u_k)+h$ and the fact that $u_k\ra u_0$ strongly in $L^2$, we obtain
\begin{align}\label{p4}
-\al\int_{\Om} (u_{k}^{+})(u_k-u_0) + \ba\int_{\Om} (u_{k}^{-})(u_k-u_0) - \int_{\Om} (f(u_{k})+h)(u_k-u_0)\ra 0.
\end{align}
From \eqref{p2}, we have $\langle E^{\prime}(u_k),
(u_k-u_0)\rangle\ra 0$. Thus using this, equation \eqref{p3} and
\eqref{p4}, we obtain
\[\int_{Q}(u_k(x)-u_k(y))((u_k-u_0)(x)-(u_k-u_0)(y))K(x-y)dxdy \ra 0\;\mbox{as}\; k\ra\infty.\]
Hence, by this and the weak convergence of $u_k$, we obtain
\[\int_{Q}|u_k(x)-u_k(y)|^2 K(x-y)dxdy \ra \int_{Q}|u_0(x)- u_0(y)|^2 K(x-y)dxdy \;\mbox{as}\; k\ra\infty.\]
It follows that $u_k\ra u_0$ strongly in $X_0$.\QED
\end{proof}
{\bf Proof of Theorem \ref{th1}}: By the Saddle point theorem we can now conclude the proof.\QED

\subsection{The Resonance Case:}
In this section, we study the problem $(P_{\la})$ in the presence of
a resonance, namely when $(\al,\ba)\in \mb R^{2}$ is an element of
Fu\v{c}ik spectrum. This kind of problem is harder to solve than the
nonresonant one and we have to impose further conditions on the
nonlinearities. We assume that $\ba=\ba(\al)$.  Many of the argument
from the previous section are still applicable. Two notable
exceptions are establishing a lower bound for $E$ on $\mc X_2$ and
proving $(PS)$. Since this case is analogous to the case
$\la=\la_{k+1}$ in this Fredholm Alternative, we should expect that
the solutions will only exist if a generalized orthogonality
condition is satisfied. Such conditions were first studied in 70s
and known as Landesman-Lazer conditions \cite{LL}. We will assume
$(GLL)$, generalized Landesman-Lazer condition.
\begin{Lemma}
If $(GLL)$ is satisfied, then $E$ is bounded below on $\mc X_2$.
\end{Lemma}
\begin{proof}
Suppose that $\{u_k\}\subset \mc X_2$ such that $E(u_k)\ra -\infty$. We will write $u_k= M_{\al,\ba}(v_k)+v_k$.
By arguments identical to those in the proof of Lemma \ref{le9} we see that no subsequence of $\{u_k\}$ lies in a set
of the form $\{u\in \mc X_2: u= M_{\al,\ba}(v)+v, \tilde{J}_{\al,\ba}(v)\geq c\|v\|_{L^2}^{2}\}$, where $c>0$.
Thus $\tilde{J}_{\al,\ba}\left(\frac{v_k}{\|v_k\|_{L^2}}\right)\ra 0$ and $\frac{v_k}{\|v_k\|_{L^2}}$ must be a
 minimizing sequence of $\tilde{J}_{\al,\ba}$. Arguments identical to those in the proof of Lemma \ref{le5}
 show that $\frac{v_k}{\|v_k\|_{L^2}}\rightharpoonup v_0$ weakly in $X_0$ and $\frac{v_k}{\|v_k\|_{L^2}}\ra v_0$ strongly in
 $L^2(\Om)$ and so, $\frac{u_k}{\|u_k\|_{L^2}}\rightharpoonup \phi$ weakly in $X_0$ and $\frac{u_k}{\|u_k\|_{L^2}}\ra \phi$ strongly in $L^2(\Om)$,
 where $\phi= M_{\al,\ba}(v)+v$ is a nontrivial eigenfunction associated with $(\al,\ba)$. By $(GLL)$,
 we know that $\lim_{k\ra \infty}\int_{\Om} (F(u_k)+hu_k) =-\infty$ and it immediately follows that
 $E(u_k)\ra \infty$, a contradiction. Hence the Lemma is proved.\QED
\end{proof}
\begin{Lemma}\label{l8}
Assume $K: \mb R^n\setminus \{0\}\ra (0,\infty)$ satisfy assumptions
$(K1)$-$(K3)$, $f$ be a bounded and continuous function and $h\in
L^{2}(\Om)$. Let $\{u_k\}_{k\in \mb N}$ be a sequence in $X_0$ such
that \eqref{p1} and \eqref{p2} hold. Then, the sequence
$\{u_k\}_{k\in\mb N}$ is bounded in $X_0$ if $(GLL)$ is satisfied.
\end{Lemma}
\begin{proof}
The first part of the proof is identical the argument in the proof
of Lemma \ref{le11}. We start with the hypothetical sequence
$\{u_k\}$ such that \eqref{p1} and \eqref{p2} hold. Suppose
$\|u_k\|_{L^2}$ is unbounded. Then argue up to the point where we
have $v_k \rightharpoonup v_0$ weakly in $X_0$, $v_k\ra v_0$
strongly in $L^2(\Om)$, where $\|v_0\|_{L^2}=1$ and $v_0$ is an
eigenfunction associated with $(\al,\ba)$. Of course, in the
resonance case this is not yet a contradiction, so further argument
is needed.

\noi Write $u_k= w_k+v_k= \tilde{w}_k+ M_{\al,\ba}(v_k)+ v_k$. Now
using the fact that $\langle J^{\prime}_{\al,\ba}(M_{\al,\ba}(v_k)+
v_k), u\rangle=0$ for all $u\in X_1$ and Lemma \ref{le1}, we have
\begin{align*}
\langle E^{\prime}(u_k), \tilde{w}_k\rangle &=\langle J^{\prime}_{\al,\ba}(u_k), \tilde{w}_k\rangle -\int_{\Om} (f(u_k)+h)\tilde{w}_k\\
&= \langle J^{\prime}_{\al,\ba}(\tilde{w}_k+ M_{\al,\ba}(v_k)+ v_k), \tilde{w}_k\rangle -\int_{\Om} (f(u_k)+h)\tilde{w}_k\\
&=\langle J^{\prime}_{\al,\ba}(\tilde{w}_k+ M_{\al,\ba}(v_k)+ v_k), \tilde{w}_k\rangle - \langle J^{\prime}_{\al,\ba}(M_{\al,\ba}(v_k)+ v_k), \tilde{w}_k\rangle\\
 &\quad \quad-\int_{\Om} (f(u_k)+h)\tilde{w}_k\\
&\leq  -\de\|\tilde{w}_k\|^{2}_{L^2} -\int_{\Om}
(f(u_k)+h)\tilde{w}_k.
\end{align*}
It follows that $\tilde{w}_k$ is bounded. Note that $\langle
J^{\prime}_{\al,\ba}(u_k), \tilde{w}_k\rangle$ must also be bounded.

Now consider
\begin{align*}
E(u_k)& = J_{\al,\ba}(u_k) -\int_{\Om} (F(u_k)+h u_k)\\
& \geq J_{\al,\ba}(\tilde{w}_k +M_{\al,\ba}(v_k)+ v_k)- J_{\al,\ba}(M_{\al,\ba}(v_k)+ v_k)-  \int_{\Om} (F(u_k)+h u_k),
\end{align*}
because $J_{\al,\ba}(M_{\al,\ba}(v_k)+ v_k)\geq 0$.
Let $g(t)= J_{\al,\ba}(M_{\al,\ba}(v_k)+ v_k+t \tilde{w}_k)$. It follows from the properties of $J_{\al,\ba}$ that $g^{\prime}(0)=0$ and $g^{\prime}(t)$ is decreasing. By the Mean value Theorem $g(1)-g(0)= g^{\prime}(c)$, for some $c\in (0,1)$. Hence $g(1)-g(0)\geq g^{\prime}(1)$. It follows that
\[J_{\al,\ba}(\tilde{w}_k +M_{\al,\ba}(v_k)+ v_k)- J_{\al,\ba}(M_{\al,\ba}(v_k)+ v_k)\geq \langle J^{\prime}_{\al,\ba}(u_k), \tilde{w}_k\rangle\]
and thus
\[E(u_k) \geq \langle J^{\prime}_{\al,\ba}(u_k), \tilde{w}_k\rangle-  \int_{\Om} (F(u_k)+h u_k). \]
\noi But the first term on the right hand side is bounded and the
second goes to $-\infty$ by $(GLL)$. This contradicts the assumption
that $E(u_k)$ is bounded. Hence $\{u_k\}$ is bounded in $L^{2}(\Om)$ ,
the remaining proof follows exactly as in the proof of Lemma
\ref{le11}.\QED

\end{proof}

\noi {\bf Proof of Theorem \ref{th2}}:  One can conclude the proof
from Lemmas \ref{le12}, \ref{l8} and Saddle point Theorem.


\noindent{\bf Acknowledgements:} The author's research is supported
by National Board for Higher Mathematics, Govt. of India, grant
number: 2/40(2)/2015/R$\&$D-II/5488.

\end{document}